\documentclass[12pt]{amsart}

\usepackage{amscd}
\usepackage{graphicx}

\newtheorem{thm}{Theorem}
\newtheorem{prop}{Proposition}
\newtheorem{cor}{Corollary}
\newtheorem{lem}{Lemma}

\theoremstyle{definition}
\newtheorem{dfn}{Definition}

\newtheorem{rem}{Remark}

\def\frak #1{\mathfrak #1}

\newcommand{\Ker}{{\sf Ker}}
\newcommand{\Coker}{{\sf Coker}}

\newcommand{\Hl}{{\sf H}}

\newcommand{\g}{{\frak g}}

\newcommand{\tp}{\otimes}
\newcommand{\e}{\wedge}

\newcommand{\Ulb}{U_{Leib}}
\newcommand{\Ults}{U_{LTS}}
\newcommand{\Ulie}{U_{Lie}}

\newcommand{\Klb}{K_{Leib}}

\newcommand{\Elb}{0 \to K_{Leib} \to U_{Leib} \to \g \to 0}
\newcommand{\Elts}{0 \to K_{LTS} \to U_{LTS} \to \g \to 0}
\newcommand{\Elie}{0 \to K_{Lie} \to U_{Lie} \to \g \to 0}

\newcommand{\chr}{{\sf char\;}}

\def\tx #1{x_{#1}\tp y_{#1}\tp z_{#1}}
\def\bx #1{\{x_{#1}, y_{#1}, z_{#1}\}}

\newcounter{ProbNumb}

\newcounter{nmr}
\def\numr{\addtocounter{nmr}{1}\noindent{\it (\arabic{nmr}) }}

\newcounter{ab}

\title{Lie triple system central extensions of Lie algebras}
\author{ R. Kurdiani }

\begin{document}

\maketitle

\section*{Introduction.}

\noindent The present paper deals with a certain interrelationship between Lie algebras, Lie triple systems and Leibniz algebras which shows up through the notion of universal central extension.

Lie algebras are classical algebraic structures very widely used in mathematics and physics. Another, less widely used but equally classical structures are Lie triple systems (see \cite{LTS1,LTS2,LTS3,LTS4}). They are also used in different areas of mathematics. Leibniz algebra is yet another algebraic structure gaining increasing importance. It is a sort of ``noncommutative'' generalization of the Lie algebra structure recently introduced in \cite{LeibnizAlgebras}. Original reason for introduction of Leibniz algebras was a new cohomology theory (see \cite{L,LP}). Later, the notion of Leibniz algebra found other applications in different areas
of mathematics.

It is well known that on each Lie algebra one may define canonically the structure of a Lie triple system. We use Leibniz algebras to study Lie triple systems arising in this way, in a similar spirit to \cite{KP} where Leibniz algebras were applied to the study of Lie algebras.

Namely, it is known that each Lie algebra is also a particular case of a Leibniz algebra. Thus speaking of universal central extensions, each perfect Lie algebra $\g$, along with its universal central extension as a Lie algebra $U_{Lie}(\g)$, also admits a \emph{Lie triple system universal central extension} $U_{LTS}(\g)$ (if we consider $\g$ as a Lie triple system) as well as the \emph{Leibniz universal central extension} $U_{Leib}(\g)$ (if we consider $\g$ as a Leibniz algebra).

Now a nice unexpected new fact is that the Lie triple system $U_{LTS}(\g)$ turns out to carry a natural Leibniz algebra structure (see Lemma \ref{LemmaLeibnStr}). This fact allows us to use properties of Leibniz and Lie algebras for Lie triple systems. Namely, the main purpose of the present article is to prove the following theorem:

\begin{thm}\label{thmMain}
For a perfect Lie algebra $\g$ over a field $F$, we have the isomorphisms
$$ \Ults(\g)\cong \Ulb(\g), \ \ \ if \ \chr F = 2, $$
$$ \Ults(\g)\cong \Ulie(\g), \ \ \ if \ \chr F \ne 2.$$
In other words, the universal central extension of $\g$
in the category of Lie triple systems is isomorphic either
to the universal central extension of $\g$ as a Lie algebra (if $\chr F \ne 2$)
or to the universal central extension of $\g$ as a Leibniz algebra (if $\chr F = 2$).
\end{thm}

In the section \ref{SectPrelim} we recall some well known facts about
Lie algebras, Leibniz algebras and Lie triple systems. In particular, in \ref{sectNATP}
we introduce non-abelian tensor product of Lie triple systems and in \ref{sectUCE}
we give general definitions and standard facts about universal central extensions
of Lie triple systems. The key lemma about natural Leibniz algebra structures on Lie triple system central extensions of perfect Lie algebras is proved in the section \ref{sectMain}. Finally, using this lemma, we obtain the proof of the main theorem (Theorem \ref{thmMain}) in section \ref{pfMain}.

Throughout this paper $F$ is a field. All vector spaces and
homomorphisms are considered over $F$ if not specified otherwise.

\section*{Acknowledgement.}

\noindent The problem solved in the article was proposed by T. Pirashvili in
a friendly conversation.

\section{Definitions and Constructions.}\label{SectPrelim}

\subsection{Recollections} Let us begin by recalling some well known definitions and
facts. Namely, we discuss two generalizations of the notion of a Lie
algebra --- the notions of a Lie triple system and a Leibniz algebra.
We also study a relationship between these two generalizations.

The notion of a Lie triple system is a classical one. The definition of Lie triple
systems can be found in many articles, but for the convenience of the reader we recall it here.

\begin{dfn}
A {\it Lie triple system} is a vector space $L$ equipped with a
ternary bracket
$$\{-,-,-\} : L\tp L\tp L\to L,$$
satisfying the following
$$ \{x,y,y\}=0, $$
$$ \{x,y,z\}+\{y,z,x\}+\{z,x,y\}=0, $$
$$ \{\{x,y,z\},a,b\}=\{\{x,a,b\},y,z\}+\{x,\{y,a,b\},z\}+\{x,y,\{z,a,b\}\}, $$
for all $x,y,z,a,b\in L$.
\end{dfn}

Any Lie algebra $\g$ can be considered as a Lie triple system.
Indeed, one can define a ternary bracket $\{-,-,-\}:\g\tp\g\tp\g\to\g$
in the following way:
$$
\{x,y,z\}=[x,[y,z]].
$$
Then $\g$ with the bracket $\{-,-,-\}$ is a Lie triple system.

An {\it ideal} in a Lie triple system $L$ is a subspace $L'$ such
that the bracket $\{x,y,z\}$ ($x,y,z\in L$) is in $L'$ whenever at
least one of the elements $x,y,z$ is in $L'$.

A Lie triple system $L$ is called {\it perfect} if $L=\{L,L,L\}$.

Leibniz algebras were introduced recently in \cite{LeibnizAlgebras}, but
they already show up in many areas of mathematics. One of the
aspects of their importance is the fact that they can be used to
study other algebraic objects. In this article we use Leibniz
algebras to study Lie triple systems and in \cite{KP} they were used
to study Lie algebras. Let us recall the definition of a Leibniz
algebra.

\begin{dfn}
A {\it Leibniz algebra} is a vector space $\frak h$ equipped with a
binary bracket
$$[-,-]:\frak h\otimes \frak h\to\frak h$$
satisfying the Leibniz identity
$$ [x,[y,z]]=[[x,y],z]-[[x,z],y], $$
for all $x,y,z\in\frak h$.
\end{dfn}

Obviously, any Lie algebra is a Leibniz algebra and conversely a
Leibniz algebra $\g$ is a Lie algebra provided $[x,x]=0$ for all
$x\in\frak g$. Note that, in this case, the Leibniz identity is
equivalent to the Jacobi identity
$$ [x,[y,z]]+[y,[z,x]]+[z,[x,y]]=0.$$


For any Leibniz algebra, one can define a ternary bracket as it was
defined above for Lie algebras. In general, we do not get a Lie
triple system in this way. To get a Lie triple system in this way, it
is necessary and sufficient that the Jacobi identity be satisfied
in the Leibniz algebra. Such Leibniz algebras will be studied in what follows.

A beauty of Leibniz algebras is shown in the following example.
Namely, one can define a Leibniz algebra structure (but not a Lie
algebra structure) on second tensor and exterior powers of a Leibniz
algebra or a Lie algebra. In fact, let $\frak g$ be a Leibniz
algebra (in particular a Lie algebra). The Leibniz algebra structure
on $\frak g\otimes\frak g$ can be defined by
$$ [x\tp y,a\tp b]= [x,[a,b]] \tp y   +  x\tp [y,[a,b]], \ \ \ x,y,a,b\in \frak g.$$
The Leibniz algebra structure on $\frak g\e\frak g$ is defined
similarly, simply replace $\tp$ by $\e$ in the above formula to get
a bracket on $\frak g\e\frak g$.

These constructions can be generalized to Lie triple systems. Let
$L$ be a Lie triple system. The brackets on $L\tp L$ and $L\e L$ are
given by
$$ [x\tp y,a\tp b]= \{x,a,b\} \tp y   +  x\tp \{y,a,b\}, \ \ \ x,y,a,b\in L,$$
$$ [x\e y,a\e b]= \{x,a,b\} \e y   +  x\e \{y,a,b\}, \ \ \ x,y,a,b\in L.$$
With these brackets $L\tp L$ and $L\e L$ are Leibniz algebras.

To go further we need the definition of an action of a Leibniz algebra
on a Lie triple system.

We say that a Leibniz algebra $\g$ {\it acts} on a Lie triple system
$L$ if we are given a map
$$L\tp\g\to L, \ x\tp g\mapsto x*g, $$
satisfying the following relations
$$(x*g)*h-(x*h)*g=x*[g,h],$$
$$\{x,y,z\}*g=\{x*g,y,z\}+\{x,y*g,z\} +\{x,y,z*g\},$$
for all $x,y,z\in L$ and $g,h\in\g$.

For any Lie triple system $L$, one can define an action of the
Leibniz algebra $L\e L$ on the Lie triple system $L$ (in the above
sense) by
$$x*(y\e z)=\{x,y,z\},$$
where $x,y,z\in L$.

Before moving to the non-abelian tensor product of Lie triple systems, let us mention an important construction
which allows us to construct a Leibniz algebra structure on a vector space in some circumstances.

Let $\g$ be a Lie algebra and let $M$ be a $\g$-module. Assume
that we are given a map of $\g$-modules $f:M\to\g$, where $\g$ is
considered as a $\g$-module via the adjoint representation. We
call such a map a {\it $\g$-equivariant map}. Then the bracket on
$M$ defined by
$$[m,n]=m*f(n),\ m,n\in M$$
gives us a Leibniz algebra structure on $M$.

\subsection{A non-abelian tensor product for Lie triple systems.}\label{sectNATP}

\noindent Let us introduce a non-abelian tensor product for Lie
triple systems. This construction allows us to describe
universal central extensions of Lie triple systems explicitly.

Let $L$ be a Lie triple system. One can equip $L\tp L\tp L$ with a
ternary bracket defined by
$$\{\tx1,\tx2,\tx3\}=\bx1\tp\bx2\tp\bx3,$$
where $x_i,y_i,z_i\in L$ for all $i=1,2,3$.

Note that this bracket does not determine a Lie triple system
structure. But one can take a quotient of $L\tp L\tp L$ to get a Lie
triple system. Let $I$ be the subspace of $L\tp L\tp L$ spanned by
the elements of the forms
$$ x\tp y\tp y, $$
$$ x\tp y\tp z + y\tp z\tp x + z\tp x\tp y ,  $$
$$ \{x,a,b\}\tp y\tp z + x\tp \{y,a,b\}\tp z + x\tp y\tp \{z,a,b\} - \{x,y,z\}\tp a\tp b  ,$$
where $x,y,z,a,b\in L$. One can see immediately that $I$ is an
ideal of $L\tp L\tp L$ and consequently the bracket introduced
above determines a bracket on the quotient $L\tp L\tp L/I$. One
can straightforwardly check that the quotient with this bracket is
a Lie triple system. We call this Lie triple system the
non-abelian tensor cube of the Lie triple system $L$ and denote it by
$L*L*L$.

Note that, the original bracket of $L$ determines a map of Lie
triple systems
$$b:L*L*L \to L.$$
We define the first homology group of $L$ (with trivial coefficients) as the
cokernel of the map $b$ and the second homology group of $L$ (again with trivial coefficients) as the kernel of the map $b$. So, we have
$$\Hl_1(L)=\Coker(b), $$
$$\Hl_2(L)=\Ker(b).$$
Explicitly
$$\Hl_1(L)=L/\{L,L,L\}, $$
$$\Hl_2(L)= \Ker (\{-,-,-\})/I,$$
where $\{-,-,-\}:L\tp L\tp L\to L$ is the original bracket of $L$
and $I$ is the subspace of $L\tp L\tp L$ defined above.

\subsection{Universal central extensions of Lie triple systems.}\label{sectUCE}

\noindent Next, let us study the universal central extension of a Lie
triple system. Let us start with the definition of central extensions of Lie
triple systems.

\begin{dfn}
A {\it central extension of a Lie triple system} $\frak g$ is a
short exact sequence of Lie triple systems
$$0\to K\to L\to \g \to 0, $$
where $\{L,L,K\}=0$.
\end{dfn}

Obviously, in a central extension we have $\{L,K,L\}=\{K,L,L\}=0 $.


In a usual way one can define the universal central extension of a
Lie triple system and prove the following

\begin{prop}
Let $L$ be a Lie triple system. The universal central extension of
$L$ exists if and only if $L$ is perfect.
\end{prop}

An explicit construction of the universal central extension is
obtained via the non-abelian tensor product which was introduced
in the previous subsection.

Let $L$ be a perfect Lie triple system. Then
$$ b:L*L*L\to L $$
is the universal central extension of $L$.

This construction shows that the kernel of the universal central
extension of $L$ is the second homology group with trivial coefficients
$\Hl_2(L)$.

\

\

One can thus consider a perfect Lie algebra $\g$ as a Lie triple system or as a Leibniz algebra and construct the
appropriate universal central extensions. One obtains three universal central extensions for each
perfect Lie algebra $\g$. Namely, let
\begin{equation}\label{equUCELeibniz}
\Elb
\end{equation}
be the universal central extension of Leibniz algebra $\g$,
\begin{equation}\label{equUCELTS}
\Elts
\end{equation}
be the universal central extension of Lie triple system $\g$ and
\begin{equation}\label{equUCELie}
\Elie
\end{equation}
be the universal central extension of Lie algebra $\g$.

It is obvious that the universal central extension of the Lie algebra $\g$ is a quotient
of the universal central extension of the Leibniz algebra (or the
Lie triple system) $\g$. There is however further unexpected relationship between these universal central extensions, which is the subject of the next section.

\section{Central Extensions of Lie Algebras Considered as Lie Triple Systems.}\label{sectMain}

\noindent In this section we work with central extensions of perfect Lie algebras considered as Lie triple systems. Here we prove a lemma which is the key to our main result. Namely, surprisingly enough, one can construct a Leibniz algebra
structure on the Lie triple system universal central extension of a perfect Lie algebra $\g$ viewed as a Lie triple system and consider it as a central extension of Leibniz algebras (see Lemma \ref{LemmaLeibnStr} below). So, the universal central extension of the Lie triple system $\g$ is a quotient of
the universal central extension of the Leibniz algebra $\g$. This
leads us to the description of the universal central extension of
the Lie triple system $\g$ via the universal central extension of
the Leibniz algebra $\g$. But the universal central extensions of
Leibniz algebras are studied very well and thus we can use this
information. The precise formulation of these facts are given in
the following lemma, which as a corollary will yield our main theorem.

\begin{lem}\label{LemmaLeibnStr}
Let $\g$ be a perfect Lie algebra and let
$$0\to K\to L\to \g \to 0 $$
be a central extension of the Lie triple system $\g$. Then there
is a Leibniz algebra structure on $L$
$$[-,-]:L\tp L\to L,$$
such that the original Lie triple system structure on $L$ is given
by the Leibniz bracket
$$ \{x,y,z\}=[x,[y,z]], \ \forall\ x,y,z\in L.$$
\end{lem}
\begin{proof}
We have to construct a Leibniz algebra structure on the Lie triple
system $L$. First we construct a $\g$-module structure on $L$,
such that the map $L\to\g$ is a $\g$-equivariant map. Then the
desired Leibniz algebra structure can be constructed as it was
described in Section \ref{SectPrelim}.

It is well known that $L\e L$ is a Leibniz algebra and acts on $L$
(see Section \ref{SectPrelim}). Since $\g$ is a perfect Lie
algebra, it can be considered as a homomorphic image of the
composition map $L\e L\to\g\e\g\to\g$. Let us denote by $Z$ the
kernel of this map. We want to show that $Z$ acts trivially on
$L$. One can easily verify that the following statements hold:

\numr for any elements $z\in Z$ and $a\in L$ we have $a*z\in K$,
where $*$ denotes the action of $L\e L$ on $L$;

\numr any element $x\in L$ can be represented as a sum of elements
$x=k+y$, where $k\in K$ and $y\in\{L,L,L\}$;

\numr the action of $L\e L$ (and in particular the action of $Z$) on
$K$ is trivial.

\noindent Thus, we need to show that $Z$ acts trivially on
$\{L,L,L\}$. This follows from the following:
$$
\begin{array}{rl}
\{a,b,c\}*z  =  &   \{a*z,b,c\}+\{a,b*z,c\}+\{a,b,c*z\}   \\
           =  &   \{k_1,b,c\}+\{a,k_2,c\}+\{a,b,k_3\}   \\
           =  &   0,                              \\
\end{array}
$$
where $k_1=a*z,k_2=b*z,k_3=c*z\in K$.

So, the action of $L\e L$ on $L$ factors through $\g$. The
desired Leibniz algebra structure on $L$ is given by the
$\g$-equivariant map $L\to\g$.

The identity
$$ \{x,y,z\}=[x,[y,z]]$$
follows immediately from the definition of the Leibniz algebra
structure on $L$.
\end{proof}

\begin{rem}\label{remLeibnStr}
Note that the Leibniz bracket which appeared in Lemma \ref{LemmaLeibnStr}
satisfies the Jacobi identity.
\end{rem}

\section{Proof of the Main Theorem}\label{pfMain}

\noindent Let $J$ be the linear subspace of $\Ulb$ generated by the elements of
the form
\begin{equation}\label{equJacobiElmts}
[x,[y,z]] + [z,[x,y]] + [y,[z,x]],  \ x,y,z\in \Ulb
\end{equation}
and let $I$ be the linear subspace of $\Ulb$ generated by the
elements of the form
$$[x,y]+[y,x], \ x,y\in \Ulb .$$
It is obvious that $I$ and $J$ are subspaces of $\Klb$ and
consequently ideals of $\Ulb$. Since the Jacobi identity is equivalent
to the Leibniz identity modulo $I$ and Leibniz identity holds in $\Ulb$,
one can immediately conclude that $J\subset I$.
More precise description is given by the following lemma.

\begin{lem}\label{lem_j=2i}
The equation
$$ J = 2 I$$
holds.
\end{lem}
\begin{proof}
The elements of the form $[x,y]+[y,x]$, $x,y\in \Ulb$,
are in the center of $\Ulb$ since their image is $0$ in $\g$.
Consequently we have $[[x,y],z] = - [[y,x],z]$ for all $x,y,z\in \Ulb$.
Using this identity together with the Leibniz identity one can show the following
$$
\begin{array}{rll}
[x,[y,z]] + [y,[z,x]] + [z,[x,y]] =& [x,[y,z]] + [[y,z],x] - [[y,x],z] + [z,[x,y]]              \\
                                  =& [x,[y,z]] + [[y,z],x] + [[x,y],z] + [z,[x,y]]              \\
                                  =& [x,[y,z]] + [[y,z],x] + [[x,z],y] + [x,[y,z]] + [z,[x,y]]  \\
                                  =& 2 [x,[y,z]] + [[y,z],x] + [[x,z],y] + [z,[x,y]]            \\
                                  =& 2 [x,[y,z]] + [[y,z],x] - [[z,x],y] + [z,[x,y]]            \\
                                  =& 2 [x,[y,z]] + [[y,z],x] - [[z,y],x]                        \\
                                  =&  2 [x,[y,z]] + [[y,z],x] + [[y,z],x]                       \\
                                  =& 2 [x,[y,z]] + 2 [[y,z],x]                                  \\
                                  =& 2 ( [x,[y,z]] + [[y,z],x] ).                                \\
\end{array}
$$
This proves the lemma since $\Ulb$ is a perfect Leibniz algebra.
\end{proof}

\begin{cor}
We have
$$ J=0, \ \ \ if \ \chr F = 2, $$
$$ J=I, \ \ \ if \ \chr F \ne 2, $$
where $F$ is the ground field.
\end{cor}

By the universal property of (\ref{equUCELeibniz}) and considering
(\ref{equUCELTS}) as a central extension of Leibniz algebras we
get a map $\Ulb\to \Ults$. Similarly we have maps $\Ulb\to \Ulie$ and $\Ults\to \Ulie$.
Again by the universal property we can assemble these maps into a
commutative diagram
$$  \Ulb\longrightarrow \Ults  $$
$$ \searrow\ \ \swarrow $$
$$    \Ulie.    $$
Let $I'$ be the image of $I$ in $\Ults$.

\begin{prop}
With the above notations we have the following isomorphisms
$$\Ults\cong \Ulb/J$$
$$\Ulie\cong \Ulb/I, \ \ \ if \ \chr F \ne 2, $$
$$\Ulie\cong \Ults/I', \ \ \ if \ \chr F \ne 2, $$
where $F$ is the ground field.
In particular
$$  0\to \Klb/J\to \Ulb/J \to \g \to 0 $$
is the universal central extension of the Lie triple system $\g$.
\end{prop}
\begin{proof}
We are going to prove the first isomorphism, the other two can be
proved similarly and are left to the reader as an exercise.

By Lemma \ref{LemmaLeibnStr} the extension (\ref{equUCELTS}) can
be considered as a central extension of the Leibniz algebra $\g$.
This implies existence of the unique map $\Ulb \to \Ults$ over $\g$. Moreover, the
elements of the form (\ref{equJacobiElmts}) are in the kernel of
this map. Recall that $J$ is an ideal of $\Ulb$. So, the quotient
$\Ulb/J$ is a Leibniz algebra and the map $\Ulb \to \Ults$ can be factored
as the composition $\Ulb \to \Ulb/J \to \Ults$. On the other hand the
ternary bracket on $\Ulb/J$ defined by the formula
$\{a,b,c\}=[a,[b,c]]$ gives a Lie triple system structure and the
extension
$$  0\to \Klb/J\to \Ulb/J \to \g \to 0 $$
can be considered as a central extension of the Lie triple system
$\g$. So, we have the unique map $\Ults \to \Ulb/J$. By the universal
properties we can conclude that the maps $\Ulb/J \to \Ults$ and
$\Ults \to \Ulb/J$ are inverse to each other. So, we have $\Ults\cong \Ulb/J$.
\end{proof}

Combining the proposition above with Lemma \ref{lem_j=2i}
one immediately gets Theorem \ref{thmMain}.


\vskip3cm



\begin{thebibliography}{AAA}

 \bibitem{LTS1} N. Jacobson, "Lie and Jordan triple systems",
                Amer. J. of Math., 71 (1949), 149–-170

 \bibitem{LTS2} W.G. Lister, "A structure theory of Lie triple systems",
                Trans. Amer. Math. Soc., 72 (1952), 217-–242

 \bibitem{LTS3} O. Loos, "Symmetric spaces", Benjamin 1969

 \bibitem{LTS4} S. Helgason, "Differential geometry, Lie groups, and symmetric spaces", Acad. Press 1978

 \bibitem{LeibnizAlgebras} J.-L. Loday, "Une version non commutative des algebres de Lie: les algebres de Leibniz",
                           Enseign. Math. (2), 39 No. 3-–4 (1993), 269-–293

 \bibitem{L} J.-L. Loday, "Cyclic homology", Springer-Verlag 1992

 \bibitem{LP} J.-L. Loday, T. Pirashvili, "Universal enveloping algebras of Leibniz algebras and (co)homology",
              Math. Ann. 296, No.1 (1993), 139--158

 \bibitem{KP} R. Kurdiani, T. Pirashvili, "A Leibniz algebra structure on the second tensor power",
              J. Lie Theory 12, No. 2 (2002), 583--596

\end{thebibliography}
\end{document}